\documentclass{article}

\usepackage{amsmath,amssymb,amsthm,amsfonts,mathrsfs,indentfirst}
\usepackage[all]{xy}

\def\co{\colon\thinspace}
\DeclareMathAlphabet{\mathsfsl}{OT1}{cmss}{m}{sl}

\newtheorem{thm}{Theorem}[section]
\newtheorem{lem}[thm]{Lemma}
\newtheorem{cor}[thm]{Corollary}
\newtheorem{prop}[thm]{Proposition}
\newtheorem*{thm*}{Theorem}

\theoremstyle{definition}

\newtheorem{rem}[thm]{Remark}

\begin{document}

\title{Two applications of twisted Floer homology}

\author{
{ \Large Yinghua AI$^{\text{\:\rm a}}$ and Yi NI$^{\text{\:\rm
b}}$} }

\date{}

\maketitle

\begin{abstract}
Given an irreducible closed $3$--manifold  $Y$, we show that its
twisted Heegaard Floer homology determines whether $Y$ is a torus
bundle over the circle. Another result we will prove is, if $K$ is
a genus $1$ null-homologous knot in an $L$--space, and the
$0$--surgery on $K$ is fibered, then $K$ itself is fibered. These
two results are the missing cases of earlier results due to the
second author.
\end{abstract}

\footnotetext{\hspace{-15pt}{\it ${\text{\rm a.}}$ School of
Mathematical
Sciences, Peking University, Beijing 100871, P. R. China}\\
{\it ${\text{\rm b.}}$ Department of Mathematics, Massachusetts
Institute of Technology, 77 Massachusetts Avenue, Cambridge, MA}}

\section{Introduction}

Heegaard Floer homology was introduced by Peter Ozsv\'ath and
Zolt\'an Szab\'o in \cite{OSzAnn1}. It associates to every closed
oriented $3$--manifold $Y$ a package of abelian groups
$\widehat{HF}(Y)$, $HF^{+}(Y)$, $HF^{-}(Y)$ and $HF^{\infty}(Y)$.
This theory also provides an invariant, called knot Floer
homology, for every null-homologous knot in a $3$--manifold
\cite{OSzKnot,Ra}.

Heegaard Floer homology turns out to be very powerful. For
example, it determines the Thurston norm of a $3$--manifold and
the genus of a knot \cite{OSzGenus}. Ghiggini \cite{Gh} and Ni
\cite{Ni1} showed that knot Floer homology detects fibered knots.
An analogue of this result for closed manifolds was proved by Ni
\cite{Ni2}, namely, Heegaard Floer homology detects whether a
closed $3$--manifold fibers over the circle when the genus of the
fiber is greater than $1$. The precise statement is as follows.

\begin{thm*}{\rm\cite{Ni2}}
Suppose $Y$ is a closed irreducible $3$--manifold, $F\subset Y$ is
a closed connected surface of genus $g>1$. Let $HF^+(Y,[F],g-1)$
denote the group
\[
    \bigoplus_{ {\bf \mathfrak{s}}\in \mathrm{Spin}^c(Y), \langle c_1({\bf \mathfrak{s}}),[F]\rangle=2g-2}HF^+(Y,{\bf \mathfrak{s}}).
\] If $HF^+(Y,[F],g-1)\cong \mathbb{Z}$, then $Y$ fibers over the circle with $F$ as a fiber.
\end{thm*}

The above theorem does not hold when $g=1$. In this case the group
$HF^+(Y,[F],0)$ is always an infinitely generated group, so its
Euler characteristic is not well-defined in the usual sense.
Ozsv\'ath and Szab\'o suggested us that one may use Heegaard Floer
homology with twisted coefficients in some Novikov ring. Some
calculations for torus bundles has been done in an earlier paper
\cite{AiP} along this line.

As in \cite{OSzAnn2}, there is a twisted Heegaard Floer homology
$\underline{HF}^+(Y;\Lambda_{\omega})$, where
$\Lambda_{\omega}$ is the universal Novikov ring equipped with a  $\mathbb{Z}[
H^1(Y;\mathbb{Z})]$--module  structure which will be defined in
Subsection~\ref{section:twist hf}. In \cite{AiP} this twisted
Heegaard Floer homology is calculated for torus bundles.

\begin{thm}{\rm\cite{AiP}}
\label{thm:hf for torus bundle} Suppose $\pi\co Y\to S^1$ is a
fiber bundle with torus fiber $F$, and $\omega \in
H^2(Y;\mathbb{Z})$ is a cohomology class such that
$\omega([F])\neq0$. Then we have
 an isomorphism of
$\Lambda$--modules
\[
    \underline{HF}^+(Y;\Lambda_{\omega}) \cong
    \Lambda.
\]
\end{thm}

The above theorem completes in some sense a result of Ozsv\'ath
and Szab\'o \cite{OSzSympl} which states  that a surface bundle
over circle has monic Heegaard Floer homology.

In the current paper, we prove the converse to the above theorem.
Our main result is:

\begin{thm}
\label{main theorem} Suppose $Y$ is a closed irreducible
$3$--manifold, $F\subset Y$ is an embedded torus. If there is a
cohomology class $\omega \in  H^2(Y;\mathbb{Z})$
 such that $\omega([F])\neq0$ and
$\underline{HF}^+(Y;\Lambda_{\omega}) \cong \Lambda$,
then $Y$ fibers over the circle with $F$ as a fiber.
\end{thm}

\begin{rem}
The proof of Theorem~\ref{main theorem} is based on Ghiggini's argument in \cite{Gh}. The only new ingredient we
introduce here is twisted coefficients. In the setting of Monopole
Floer homology, a corresponding version of this theorem was proved
in \cite[Theorem~42.7.1]{KM}, following Ghiggini's argument.
\end{rem}

Besides Theorem~\ref{main theorem}, we give an application of
Theorem~\ref{thm:hf for torus bundle}.

\begin{thm}\label{fiber knot}
Suppose $Y$ is an $L$--space, $K \subset Y$ is
a genus $1$ null-homologous knot. If the $0$--surgery on $K$
fibers over $S^1$, then $K$ itself is a fibered knot.
\end{thm}

The case where $K$ has genus greater than $1$ has been proved in
\cite[Corollary 1.4]{Ni1}.

This paper is organized as follows. In Section 2, we collect some
preliminary results on Heegaard Floer homology with an emphasis on
twisted coefficients. In Section 3, we prove a key proposition
which relates the Euler characteristic of $\omega$--twisted
Heegaard Floer homology with Turaev torsion. With the help of this
proposition, we can prove a homological version of the main
theorem. In Section~4, we give an proof of Theorem~\ref{main
theorem} following Ghiggini's argument. In Section 5, we prove
Theorem~\ref{fiber knot} by using Theorem~\ref{thm:hf for torus
bundle} and the exact sequence for $\omega$--twisted Floer
homology from \cite{AiP}.

\vspace{5pt}\noindent\textbf{Acknowledgements.} We are grateful to
Peter Ozsv\'ath and Zolt\'an Szab\'o for suggesting us to use
twisted Heegaard Floer homology for the genus $1$ case. In
particular, we thank Peter for helpful conversations and
encouragements, and for providing the ideas to prove
Proposition~\ref{prop:eulerchar}. The first author would also like
to thank Thomas Peters for many helpful discussions during their
joint work \cite{AiP}.

This work was carried out while the two authors were at Columbia
University. Both authors are grateful to the Columbia math
department for its hospitality.

The first author is supported by China Scholarship Council. The
second author is partially supported by an AIM Five-Year
Fellowship and NSF grant number DMS-0805807.

\section{Preliminaries on Heegaard Floer Homology}
We review some of the constructions in Heegaard Floer homology
which will be used throughout this paper. The details can be found
in \cite{OSzAnn2,OSzKnot,OSzGenus,OSzCont,AiP}.

\subsection{Heegaard Floer homology with twisted coefficients}
\label{section:twist hf}

Let $Y$ be a closed oriented $3$--manifold and $\mathfrak{t}$ be a
Spin$^c$ structure over $Y$. Ozsv\'ath and Szab\'o \cite{OSzAnn2}
defined a universally twisted chain complex
$\underline{CF}^{\circ}(Y,\mathfrak{t})$ with coefficients in the
group ring $\mathbb{Z}[ H^1(Y;\mathbb{Z})]$. Its homology
$\underline{HF}^{\circ}(Y,\mathfrak{t})$ is an invariant of the
pair $(Y,\mathfrak{t})$. Furthermore, given any module $M$ over
$\mathbb{Z}[ H^1(Y;\mathbb{Z})]$, they defined Floer homology with
coefficients twisted by $M$:
\[
\underline{HF}^{\circ}(Y,\mathfrak{t};M)=
H_{\ast}(\underline{CF}^{\circ}(Y,\mathfrak{t})
\otimes_{\mathbb{Z}[ H^1(Y;\mathbb{Z)}]}M)
\]
This construction recovers the ordinary Heegaard Floer homology if
we take $M$ to be the trivial $\mathbb{Z}[
H^1(Y;\mathbb{Z})]$--module $\mathbb{Z}$.

A special twisted Floer homology is used to investigate torus
bundles over the circle  in \cite{AiP}. Consider the universal Novikov ring \cite[Section 11.1]{MD}
  \[
  \Lambda = \left\{\sum_{r \in \mathbb{R}}a_r t^r\bigg|a_r\in\mathbb{R},
  \;\#\{a_r|a_r\ne0, r \leq c\}<\infty\quad \text{\rm for any $c\in\mathbb R$}
  \right\}.
  \]

$\Lambda$ itself is a field. Given a
cohomology class $\omega \in  H^2(Y;\mathbb R) $, there is a group
homomorphism $$
\begin{array}{ccl}
 H^1(Y;\mathbb{Z})&\to &\mathbb R\\
  h  &\mapsto&
\langle h \cup \omega,[Y]\rangle,
\end{array}$$
which then induces a ring homomorphism
$$\begin{array}{ccl}
 \mathbb{Z}[H^1(Y;\mathbb{Z})]&\to &\Lambda\\
  \sum a_{h}\cdot h  &\mapsto&
\sum a_{h}\cdot t^{\langle h \cup \omega,[Y]\rangle}.
\end{array}$$

In this way we can equip $\Lambda$ with an induced
$\mathbb{Z}[ H^1(Y;\mathbb{Z})]$--module structure. We denote this
module by $\Lambda_{\omega}$.

 This
$\mathbb{Z}[ H^1(Y;\mathbb{Z})]$--module $\Lambda_{\omega}$ gives
rise to a  twisted Heegaard Floer homology
$\underline{HF}^+(Y;\Lambda_{\omega})$ which is called the
$\omega$--twisted Heegaard Floer homology. More precisely, it is
defined as follows. (For more details, see \cite{OSzGenus} and
\cite{AiP}.) Choose an admissible  pointed  Heegaard diagram
$(\Sigma,\mbox{\boldmath${\alpha}$}, \mbox{\boldmath${\beta}$},z)$
for $Y$. Every Whitney disk $\phi:\mathbb{D}^2 \to
\mathrm{Sym}^g(\Sigma)$ gives rise to a two-chain in $Y$. Let
$\eta$ be a closed $2$--cochain that represents $\omega$. The
evaluation of $\eta$ on $\phi$ is denoted $\int_{\phi}\eta$. Take
the $\mathbb Z[ H^1(Y;\mathbb Z)]$--module freely generated by all
the pairs $[\mathbf x,i]$ (where $\mathbf x \in \mathbb{T}_\alpha
\cap \mathbb{T}_\beta$ and the integer $i \geq 0$), its  tensor
product with the module $\Lambda_{\omega}$ is the
$\omega$--twisted chain complex $\underline{CF}^+(Y;
\Lambda_{\omega})$. We define the differential on the complex by
the formula:
\[
    \underline{\partial}^+[{\mathbf x},i] = \sum_{{\mathbf y}\in \mathbb{T}_\alpha\cap \mathbb{T}_\beta}\sum_{ \{\phi\in\pi_2({\bf x}, {\bf y})|\mu(\phi)=1\}} \#\widehat{\mathcal{M}}(\phi)[{\bf y},i-n_z(\phi)]\cdot t^{\int_{\phi}\eta}.
\]
The homology of this chain complex only depends on the cohomology
class $\omega$. We call this homology  the $\omega$--twisted
Heegaard Floer homology $\underline{HF}^+(Y;\Lambda_{\omega})$.
This homology also has a $\Lambda$--module structure. Since
$\Lambda$ is a field, a $\Lambda$--module is actually a vector
space over $\Lambda$.

Theorem~\ref{thm:hf for torus bundle} shows that this
$\omega$--twisted Floer homology is very simple for torus bundles
over $S^1$ when $\omega$ evaluates non-trivially on the fiber. The
key ingredients in the proof are an exact sequence for
$\omega$--twisted Heegaard Floer homology and the adjunction
inequality.

\begin{thm}{\rm\cite{AiP}}\label{thm:exact seq}
Let $K \subset Y$ be a framed knot in a $3$--manifold $Y$ and
$\gamma \subset Y-K$ be a simple closed curve in the knot
complement.  For every rational number $r$, $\gamma$ is a curve in
the surgery manifold $Y_r(K)$. Let
$\omega_r=\mathrm{PD}(\gamma)\in
 H^2(Y_r(K);\mathbb{R})$ and $\omega=\mathrm{PD}(\gamma) \in H^2(Y;\mathbb{R})$. Then we have the following exact sequence:
\begin{equation}\label{Seq:01}
\begin{xymatrix}{ \underline{HF}^+(Y;\Lambda_{\omega}) \ar[rr]
&&\underline{HF}^+(Y_0(K);\Lambda_{\omega_0}) \ar[dl]\\
&\underline{HF}^+(Y_1(K);\Lambda_{\omega_1}).\ar[ul]&}
\end{xymatrix}
\end{equation}
The maps in the above sequence are induced from cobordism.
\end{thm}

The proof  of the above theorem  in \cite{AiP} can be applied to
general integer surgeries. As in \cite[Theorem~9.19]{OSzAnn2},
suppose $\mathfrak{s}$ is a Spin$^c$ structure on $Y$,
$\mathfrak{s}_k$ is any one of the $p$\; Spin$^c$ structures on
$Y_p(K)$ which are Spin$^c$--cobordant to $\mathfrak{s}$, $Q\co
\mathrm{Spin}^c(Y_0)\to\mathrm{Spin}^c(Y_p)$ is a surjective map.
Let
\[
\underline{HF}^+(Y_0(K),[\mathfrak{s}_k];\Lambda_{\omega_0})=\bigoplus_{\mathfrak{t}
\in
\mathrm{Spin}^c(Y_0(K)),Q(\mathfrak{t})=\mathfrak{s}_k}\underline{HF}^+(Y_0(K),\mathfrak{t};\Lambda_{\omega_0}).
\] The following exact sequence  will
be used in the proof of Theorem \ref{fiber knot}.

\begin{thm}
\label{thm:integral surgery} Notation is as in
Theorem~\ref{thm:exact seq}. For each positive integer $p$, the
following sequence
$$
\hspace{-10pt}\begin{xymatrix}{
\hspace{20pt}\underline{HF}^+(Y,\mathfrak{s};\Lambda_{\omega})
\hspace{20pt}\ar[r]&&\hspace{-60pt}
\underline{HF}^+(Y_0(K),[\mathfrak{s}_k];\Lambda_{\omega_0})
\ar[dl]\\
&\hspace{-10pt}\underline{HF}^+(Y_p(K),\mathfrak{s}_k;\Lambda_{\omega_p})
\ar[ul]_{F^+_3}&}
\end{xymatrix}
$$
is exact. Here the map $F^+_3$ is induced by a two-handle
cobordism connecting $Y_p(K)$ to $Y$.
\end{thm}

Combine Theorem \ref{thm:exact seq} and \cite[Proposition
2.2]{AiP}, we can easily get the following corollary, which is
used  implicitly in \cite{AiP} to prove Theorem \ref{thm:hf for torus bundle}.

\begin{cor}
\label{cor:reglue} Suppose  $F$ is an embedded torus in a closed
manifold $Y$ and $\omega \in  H^2(Y;\mathbb{Z})$ is a cohomology
class such that $\omega([F]) \neq 0$. Let $Y'$ be the manifold
obtained by cutting open $Y$ along $F$ and regluing by a
self-homeomorphism of $F$. There is a cobordism $W\co Y \to Y'$
obtained by adding $2$--handles along knots in $F$. Take a closed
curve $\gamma \subset Y$  missing the attaching regions of the
$2$--handles  and $\mathrm{PD}(\gamma)=\omega \in
H^2(Y;\mathbb{Z})$. As in Theorem~\ref{thm:exact seq}, $\gamma$
also determines a cohomology class $\omega'$ in $Y'$ by setting
$\omega'=\mathrm{PD}_{Y'}(\gamma)$. Then the map induced by
cobordism is an isomorphism
\[
\underline{HF}^+(Y';\Lambda_{\omega'}) \cong
\underline{HF}^+(Y;\Lambda_{\omega}).
\]
\end{cor}

\subsection{Knot Floer homology}

Suppose $K \subset Y$ is a null-homologous knot in a
rational homology $3$--sphere, $F$ is a fixed Seifert surface. There is  a compatible doubly pointed Heegaard diagram
$(\Sigma,\mbox{\boldmath$\alpha$},\mbox{\boldmath$\beta$},w,z)$ for the knot $K$ as in \cite{OSzKnot}. This gives rise to  a map from intersection points
between the two tori $\mathbb{T}_{\alpha},\mathbb{T}_{\beta}$ to
relative Spin$^c$ structures on $Y-K$
\[
s_{w,z}\co\mathbb{T}_{\alpha} \cap \mathbb{T}_{\beta} \to
\underline{\mathrm{Spin}^c}(Y,K)=\mathrm{Spin}^c(Y_0(K)).
\]

For each Spin$^c$ structure $\mathfrak{s}$ on $Y$, the knot chain
complex $$C(\mathfrak{s})=CFK^{\infty}(Y,K;\mathfrak{s})$$ is a
free abelian group generated by $[\mathbf x,i,j] \in
(\mathbb{T}_{\alpha} \cap \mathbb{T}_{\beta}) \times \mathbb{Z}
\times \mathbb{Z}$, such that $s_{w,z}(\mathbf x)$ extends
$\mathfrak{s}$ and
\[
\frac{\langle c_1(s_{w,z}(\mathbf x)),[\widehat{F}]
\rangle}{2}+(i-j)=0,
\]
and endowed  with the differential
\[
\partial [\mathbf x,i,j]=\sum_{\mathbf y \in \mathbb{T}_{\alpha} \cap \mathbb{T}_{\beta}}
\sum_{\{\phi \in \pi_2(\mathbf x,\mathbf y)|\mu(\phi)=1\}} \#
\widehat{\mathcal{M}}(\phi)[{\bf y},i-n_w(\phi),j-n_z(\phi)].
\]
This complex is given a filtration function $\mathcal{F}[{\bf
x},i,j]=(i,j)$. The forgetful map $[{\bf x},i,j] \mapsto [{\bf
x},i]$ induces an isomorphism between $C(\mathfrak{s})$ and
$CF^{\infty}(Y,\mathfrak{s})$, sending
$B^+_{\mathfrak{s}}:=C(\mathfrak{s}){\{i \geq 0\}}$ isomorphically
to $CF^+(Y,\mathfrak{s})$. For each integer $d$, we define
\[
\widehat{HFK}(Y,K,\mathfrak{s};d)=
H_{\ast}\big(C(\mathfrak{s}){\{i=0,j \le
d\}}/C(\mathfrak{s}){\{i=0,j \le d-1\}},\partial\big).
\]

Define $A_{\mathfrak{s},k}^+=C(\mathfrak{s}){\{\text{max}(i,j-k)
\geq 0\}}$. There are two canonical chain maps
$v_{\mathfrak{s},k}^+\co A_{\mathfrak{s},k}^+ \to
B_{\mathfrak{s}}^+$  and $h_{\mathfrak{s},k}^+\co
A_{\mathfrak{s},k}^+ \to B_{\mathfrak{s}}^+$. The map
$v_{\mathfrak{s},k}^+$ is the projection onto $C(\mathfrak{s}){\{i
\geq 0\}}$, while $h_{\mathfrak{s},k}^+$ is the projection onto
$C(\mathfrak{s}){\{j \geq k\}}$, followed by the identification
with $C(\mathfrak{s}){\{j \geq 0 \}}$ induced by the
multiplication by $U^k$, followed by the chain homotopy
equivalence from $C(\mathfrak{s}){\{j \geq 0\}}$ to
$C(\mathfrak{s}){\{i \geq 0\}}$.

\begin{thm}[Ozsv\'ath--Szab\'o,\cite{OSzFour,OSzIntSurg}]
\label{large surgery} Let $K \subset Y$ be a null-homologous knot
in a rational homology sphere. There is an integer $N$ with the
property that for all $m \geq N$ and all $t \in
\mathbb{Z}/{m\mathbb{Z}}$, $CF^+(Y_m(K),\mathfrak{s}_t)$ is
represented by the chain complex
$A_{\mathfrak{s},k}^+=C(\mathfrak{s}){\{\text{max}(i,j-k) \geq
0\}}$ where $k \equiv t (\text{mod } m)$ and $|k| \leq
\frac{m}{2}$, in the sense that there are isomorphisms
\[
\Psi_{m,k}^+\co CF^+(Y_m(K),\mathfrak{s}_t) \to
A_{\mathfrak{s},k}^+.
\]
Moreover, if $\mathfrak{x}_k$ and $\mathfrak{y}_k$ denote the {\rm
Spin}$^c$ structures over the cobordism
$$W'_m(K)\co Y_m(K) \to Y$$
with
\[
\langle c_1(\mathfrak{x}_k),[\widehat{F}] \rangle +m=2k  \text{
and } \langle c_1(\mathfrak{y}_k),[\widehat{F}] \rangle -m=2k,
\]
respectively. Then $v_k^+$ and $h_k^+$ correspond to the maps
induced by the cobordism $W'_m(K)$ endowed with the Spin$^c$
structures  $\mathfrak{x}_k$ and $\mathfrak{y}_k$, respectively.
\end{thm}

\subsection{Ozsv\'ath--Szab\'o contact invariant in twisted Floer homology}

Ozsv\'ath and Szab\'o \cite{OSzCont} defined an invariant $c(\xi)
\in \widehat{HF}(-Y)$ for every contact structure $\xi$ on a
closed $3$--manifold $Y$. It is defined up to sign and lies in the
summand $\widehat{HF}(-Y,\mathfrak{t}_{\xi})$ corresponding to the
canonical Spin$^c$ structure $\mathfrak{t}_{\xi}$  associated to
$\xi$.

In \cite{OSzGenus}, Ozsv\'ath and Szab\'o point out that the
contact invariant can be defined for twisted Heegaard Floer
homology. In fact, for any module $M$ over $\mathbb{Z}[
H^1(Y;\mathbb{Z})]$ one can get an element
\[
c(\xi;M) \in
\underline{\widehat{HF}}(-Y,\mathfrak{t}_{\xi};M)/\mathbb{Z}[
H^1(Y,\mathbb{Z})]^{\times}.
\]
This is an element $c(\xi;M) \in
\underline{\widehat{HF}}(-Y,\mathfrak{t}_{\xi};M)$ which is
well-defined up to an overall multiplication by a unit in the
group ring $\mathbb{Z}[ H^1(Y;\mathbb{Z})]$. Let $c^+(\xi;M)$
denote the image of $c(\xi;M)$ under the natural map $
\underline{\widehat{HF}}(-Y,\mathfrak{t}_{\xi};M) \to
\underline{HF}^+(-Y,\mathfrak{t}_{\xi};M)$.

In particular, for $\omega$--twisted  Heegaard Floer homology
$\underline{HF}^+(Y;\Lambda_{\omega})$, one gets a contact
invariant $c^+(\xi;\Lambda_{\omega}) \in
\underline{HF}^+(-Y,\mathfrak{t}_{\xi};\Lambda_{\omega}) $. It is
well-defined up to multiplication by  a term $\pm t^n$ for some $n
\in \mathbb{R}$. There is also a non-vanishing theorem for weakly
 fillable contact structures in this $\omega$--twisted
version.

\begin{thm}{\rm\cite[Theorem~4.2]{OSzGenus}}
\label{thm:non-vanishing} Let $(W,\Omega)$ be a weak filling
of a contact manifold $(Y,\xi)$, where $\Omega \in H^2(W;\mathbb{R})$ is the symplectic $2$--form. Then the contact invariant
$c^+(\xi;\Lambda_{\Omega|_Y})$ is non-trivial.
\end{thm}

\begin{rem}
In \cite{OSzGenus} Oszv\'ath and Szab\'o proved that if
$(W,\Omega)$ is  a symplectic filling of a contact structure
$(Y,\xi)$, then the contact invariant $c^+(\xi;[\Omega|_Y])$ is
non-trivial. That proof works for our  $\omega$--twisted version
as well given the following description of the $\omega$--twisted
contact invariant.

Take an open book decomposition $(Y,K)$ compatible with the
contact structure $\xi$. After positive stabilization, we can
assume that the open book has connected binding and genus $g>1$.
Suppose $W\co Y \to Y_0(K)$ is the corresponding Giroux
$2$--handle cobordism \cite{Gir}, and $\beta \in  H^2(W;\mathbb
R)$  is any cohomology class on $W$ which extends
 $\omega$. Then we have
 $$\underline{HF}^+(-Y_0(K),g-1;\Lambda_{\beta|_{-Y_0(K)}}) \cong \Lambda.$$
Moreover, the contact invariant $c^+(\xi;\Lambda_{\omega})$
is equal to the image of $1 \in \Lambda$ under the map
induced by cobordism
$$\underline{F}^+_{W;[\beta]}\co\underline{HF}^+(-Y_0(K),g-1;\Lambda_{\beta|_{-Y_0(K)}}) \to \underline{HF}^+(-Y,\mathfrak{t}_{\xi};\Lambda_{\omega}).$$

In the untwisted case, such description for the contact invariant
is proved in \cite[Proposition~3.1]{OSzCont}. In that proof, one
constructed a Heegaard diagram for $Y_0(K)$  which is admissible
with respect to all the Spin$^c$ structures $\mathfrak{t}$ such that
$$\langle c_1(\mathfrak{t}), [\widehat{F}] \rangle =2g-2.$$
In this diagram there are only $2$ intersection points
representing Spin$^c$ structures satisfying the above restriction.
Notice we can also use this Heegaard diagram to compute the
$\omega$--twisted Heegaard Floer homology.
 So this argument can be used to prove the above statements.
\end{rem}

The following corollary is a little modification of \cite[Lemma
2.3]{Gh}.

\begin{cor}
\label{cor:nonvanishing special} Suppose $Y$ is a closed,
connected, oriented $3$--manifold with $b_1(Y)=1$, and $F$ is a
closed surface in $Y$. Suppose $\xi$ is a weakly fillable contact
structure on $Y$, such that $\xi|_F$ is $C^0$--close to the
oriented tangent plane field of $F$, and $\omega \in
H^2(Y;\mathbb{R})$ is a cohomology class such that
$\omega([F])>0$. Then the $\omega$--twisted contact invariant
$c^+(\xi;\Lambda_{\omega})$ is non-zero.
\end{cor}
\begin{proof}
Let $(W,\Omega)$ be a weak filling of the contact structure
$(Y,\xi)$, $[\Omega]\in H^2(W)$ be the cohomology class
represented by the closed $2$--form $\Omega$. By the definition of
weak fillability, $\Omega$ is positive on each plane in $\xi$, so
$\langle[\Omega]|_Y,[F]\rangle>0$. Since $b_1(Y)=1$ and
$\omega([F])>0$, we may assume $[\Omega]|_Y=k \omega \in
 H^2(Y;\mathbb{R})$ for some positive real number $k$. By
Theorem~\ref{thm:non-vanishing}, $c^+(\xi;\Lambda_{\Omega})
\neq 0$. Notice that the map $t \mapsto t^k$ induces an
isomorphism between chain complexes
$$\underline{CF}^+(-Y;\Lambda_{\omega}) \to
\underline{CF}^+(-Y;\Lambda_{\Omega}),$$ it thus induces an
isomorphism between  homology groups
$$\underline{HF}^+(-Y;\Lambda_{\omega}) \to
\underline{HF}^+(-Y;\Lambda_{\Omega}).$$  Under this
isomorphism $c^+(\xi;\Lambda_{\omega})$ is taken to
$c^+(\xi;\Lambda_{\Omega})$. This forces the contact
invariant $c^+(\xi;\Lambda_{\omega})$ to be non-zero.
\end{proof}

Given a contact manifold $(Y,\xi)$ and a Legendrian knot $K\subset
Y$, we can do contact $(+1)$--surgery to produce a new contact
manifold $(Y_1(K),\xi')$. In \cite{OSzCont,LS}, it is showed that
the untwisted contact invariant behaves well with respect to
contact $(+1)$--surgery.  Similarly, the $\omega$--twisted contact
invariants are related as follows.

\begin{prop}{\rm\cite[Theorem 2.3]{LS}}
\label{prop:contact surgery} Suppose $(Y',\xi')$ is obtained from
$(Y,\xi)$ by contact $(+1)$--surgery along a Legendrian knot $K
\subset Y$. Suppose $\gamma \subset Y-K$ is a closed curve, denote
$\omega=PD(\gamma) \in H^2(Y;\mathbb{R})$. View $\gamma$ as a
curve in the surgery manifold $Y'$ and denote its Poincar\'e dual
by $\omega' \in H^2(Y';\mathbb{R})$. Let $-W$ be the cobordism
obtained from $Y \times I$ by adding a $2$--handle along $K$ with
$+1$ framing  and with orientation reversed. Then
\[
\underline{F}^+_{-W;\mathrm{PD}(\gamma \times
I)}(c^+(Y,\xi;\Lambda_{\omega}))=c^+(Y',\xi';\Lambda_{\omega'}).
\]
\end{prop}

\section{Euler characteristic of $\omega$--twisted Floer homology}

The goal of this section is to prove a homological version of
Theorem~\ref{main theorem}. In order to do so, we first study the
Euler characteristic of twisted Heegaard Floer homology.

In \cite{OSzAnn2} Ozsv\'ath and Szab\'o prove that the Euler
characteristic of $HF^+(Y,\mathfrak{s})$ is equal to the Turaev
torsion when $\mathfrak{s}$ is a non-torsion Spin$^c$ structure.
More precisely, they prove the following:

\begin{thm*}{\rm\cite[Theorem 5.11, Theorem 5.2]{OSzAnn2}}
If $\mathfrak{s}$ is a non-torsion {\rm Spin}$^c$ structure over
an oriented $3$--manifold $Y$ with $b_1(Y) \geq 1$, then
$HF^+(Y,\mathfrak{s})$ is finitely generated and
\[
\chi(HF^+(Y,\mathfrak{s}))=\pm T(Y,\mathfrak{s}).
\]
When $b_1(Y)=1$, the Turaev torsion in the above equation is
defined with respect to the chamber containing
$c_1(\mathfrak{s})$.
\end{thm*}

\begin{rem}
The proof of the above theorem can be modified to show that when
$\mathfrak{s}$ is a non-torsion Spin$^c$ structure, the
$\omega$--twisted Heegaard Floer homology
$\underline{HF}^+(Y,\mathfrak{s};\Lambda_{\omega})$ is a
finitely generated vector space over $\Lambda$, and
\[
\chi(\underline{HF}^+(Y,\mathfrak{s};\Lambda_{\omega}))=\pm
T(Y,\mathfrak{s}).
\]
\end{rem}

For a torsion Spin$^c$ structure $\mathfrak{s}$, however,
$HF^+(Y,\mathfrak{s})$ is infinitely generated. Ozsv\'ath and
Szab\'o \cite{OSzAnn2} introduce truncated Euler characteristic
for torsion Spin$^c$ structures and prove when $b_1(Y)=1 \text{ or
} 2$, it is related to the Turaev torsion.

When we use $\omega$--twisted Heegaard Floer homology, the
situation is much simpler. In fact we have the following:

\begin{prop}\label{prop:eulerchar}
Suppose $\mathfrak{s}$ is a torsion Spin$^c$ structure over an
oriented $3$--manifold $Y$ with $b_1(Y) \geq 1$. Then for any
non-zero cohomology class  $\omega \in  H^2(Y;\mathbb{R})$, the
twisted Floer homology
$\underline{HF}^+(Y,\mathfrak{s};\Lambda_{\omega})$ is a finitely
generated vector space over $\Lambda$ and
\[
\chi(\underline{HF}^+(Y,\mathfrak{s};\Lambda_{\omega}))=\pm
T(Y,\mathfrak{s}).
\]
Here the Euler characteristic is taken by viewing
$\underline{HF}^+(Y,\mathfrak{s};\Lambda_{\omega})$ as a vector
space over $\Lambda$.
\end{prop}

Notice when $b_1(Y)=1$ and $\mathfrak{s}$ is a torsion Spin$^c$
structure, $T(Y,\mathfrak{s})$ does not depend on the choice of a
chamber, see \cite{T}. The following lemma  is an analogue of
 \cite[Corollary 8.5]{JM}.

\begin{lem}
\label{lem:hf infinity} Let $\mathfrak{s}$ be a torsion {\rm
Spin}$^c$ structure on a $3$--manifold $Y$, $\omega \in
H^2(Y;\mathbb{R})$ be a non-zero cohomology class. Then
\[
\underline{HF}^{\infty}(Y,\mathfrak{s};\Lambda_{\omega})\cong 0.
\]
\end{lem}
\begin{proof}
It is shown in \cite[Theorem~10.12]{OSzAnn2} that for each
torsion Spin$^c$ structure $\mathfrak{s}$ , the universally
twisted Heegaard Floer homology
 \[
 \underline{HF}^{\infty}(Y,\mathfrak{s}) \cong
 \mathbb{Z}[U,U^{-1}],
\]
and it is  a trivial $\mathbb{Z}[ H^1(Y;\mathbb{Z})]$--module.
Then there is a universal coefficients spectral sequence
converging to the $\omega$--twisted Floer homology
$\underline{HF}^{\infty}(Y,\mathfrak{s};\Lambda_{\omega})$, and
its $E_2$ term is given by
$$
\mathrm{Tor}_{\mathbb{Z}[
H^1(Y;\mathbb{Z})]}^i(\underline{HF}^{\infty}_{j}(Y,\mathfrak{s});\Lambda_{\omega}).
$$

Notice $\underline{HF}^{\infty}_{j}(Y,\mathfrak{s})=0 \text{\rm \
or } \mathbb{Z}$. There is a free resolution of $\mathbb{Z}$ as a
module over $\mathbb{A}=\mathbb{Z}[ H^1(Y;\mathbb{Z})]$, given by
\[
\otimes_{i=1}^{b_1(Y)}(\mathbb{A}
\stackrel{e^{h_i}-1}{\longrightarrow} \mathbb{A}),
\]
where $h_i$ is a basis for $ H^1(Y;\mathbb{Z})$, and $e^{h_i}$ is
the corresponding element in the group ring $\mathbb{A}$ (see
\cite[Lemma~2.3]{OSzSympl}). So the $E_2$ term of the above
spectral sequence is calculated by the homology of
\[
\otimes_{i=1}^{b_1(Y)}(\Lambda_{\omega}
\stackrel{t^{d_i}-1}{\longrightarrow} \Lambda_{\omega}),
\]
where $d_i=\langle h_i \cup \omega,[Y] \rangle$. By assumption
$\omega \neq 0 \in  H^2(Y;\mathbb{R})$, at least one of $d_i$ is
non-zero. For this $i$, the map
\[
\Lambda_{\omega} \stackrel{t^{d_i}-1}{\longrightarrow}
\Lambda_{\omega}
\]
is an isomorphism, hence the corresponding complex has zero
homology. From this we see that all the $E_2$ terms are $0$, so
are the $E_{\infty}$ terms. This proves the lemma.
\end{proof}

Recall that for a torsion Spin$^c$ structure $\mathfrak{s}$, there
is an absolute $\mathbb{Q}$--grading on
$\underline{HF}^+(Y,\mathfrak{s})$ which lifts the relative
$\mathbb{Z}$--grading defined in \cite{OSzAnn2}, see
\cite{OSzFour}. This is also the case for our $\omega$--twisted
Floer homology
$\underline{HF}^+(Y,\mathfrak{s};\Lambda_{\omega})$.
Suppose the absolute grading is supported in $\mathbb{Z}+d$ for
some constant $d \in \mathbb{Q}$. The Euler characteristic of
$\underline{HF}^+(Y,\mathfrak{s};\Lambda_{\omega})$ is
defined to be
\[
\chi(\underline{HF}^+(Y,\mathfrak{s};\Lambda_{\omega}))=
\sum_{n \in \mathbb{Z}}(-1)^n \mathrm{rank}\:
\underline{HF}^+_{d+n}(Y,\mathfrak{s};\Lambda_{\omega}).
\]
Notice our $d$ is unique up to adding an integer, so the Euler
characteristic is defined up to sign.

This absolute $\mathbb{Q}$--grading and Lemma~\ref{lem:hf
infinity} lead to the following corollary.

\begin{cor}
\label{cor:high grading} Let $\mathfrak{s}$ be a torsion {\rm
Spin}$^c$ structure on a $3$--manifold $Y$, $\omega\in
H^2(Y;\mathbb{R})$ be a non-zero cohomology class. Then for all
sufficiently large $N \in \mathbb{Z}$,
\[
\underline{HF}^+_{d+N}(Y,\mathfrak{s};\Lambda_{\omega})=0.
\]
\end{cor}
\begin{proof}
The elements in
$\underline{HF}^-(Y,\mathfrak{s};\Lambda_{\omega})$ have
absolute $\mathbb{Q}$--grading bounded from above. So by the exact
sequence relating $\underline{HF}^-,\underline{HF}^{\infty} \text{ and } \underline{HF}^+$ we see that
for all sufficiently large $N \in \mathbb{Z}$,
\[
 \underline{HF}^+_{d+N}(Y,\mathfrak{s};\Lambda_{\omega}) \cong
 \underline{HF}^{\infty}_{d+N}(Y,\mathfrak{s};\Lambda_{\omega}),
\]
which  is zero by Lemma \ref{lem:hf infinity}.
\end{proof}

\begin{proof}[Proof of Proposition \ref{prop:eulerchar}](Compare the proof of \cite[Theorem~5.2]{OSzAnn2}.)

From Corollary~\ref{cor:high grading},
$\underline{HF}^+(Y,\mathfrak{s};\Lambda_{\omega})$ is a finitely
generated vector space over $\Lambda$. So we can talk about its
Euler characteristic. As in \cite[Section~5.3]{OSzAnn2} we can
construct a Heegaard diagram
$(\Sigma,\mbox{\boldmath$\alpha$},\mbox{\boldmath$\beta$},z)$ for
$Y$ such that there is a periodic domain $P_1$ containing
$\alpha_1$ with multiplicity one in its boundary. Extend $\{P_1\}$
to a basis $\{P_1,P_2,\dots,P_b\}$ for periodic domains such that
$P_2,\dots,P_b$ does not contain $\alpha_1$ in their boundaries.
This can be achieved by adding proper multiples of $P_1$ to each
$P_i$. Choose a set of dual simple closed curves $\{a_i^{\ast}\}$
for $\{\alpha_i\}$, namely, $a_i^{\ast}$ meets $\alpha_i$
transversely in a single point and misses all other $\alpha_j$.
Wind $\alpha_1$ along $a_1^{\ast}$ $n$ times and put the base
point $z$ in this winding region, to the right of $a_1^{\ast}$ and
in the $\frac{n}{2} \text{th}$ subregion of the winding region.
(See \cite[Figure~6]{OSzAnn2}.) Wind $\alpha_2,\dots,\alpha_g$
along $a_2^{\ast},\dots,a_g^{\ast}$ sufficiently many times, such
that any nontrivial linear combination of $P_2,\dots,P_n$ has both
large positive and negative local multiplicities as in
\cite[Lemma~5.4]{OSzAnn1}.

When $n$ is sufficiently large, by \cite[Lemma~5.4]{OSzAnn2}, if
an intersection point represents the fixed torsion Spin$^c$
structure $\mathfrak{s}$, then its $\alpha_1$--component must lie
in the winding region corresponding to $\alpha_1$. These
intersection points are partitioned into two subsets according to
whether they lie to the left or right side of $a_1^{\ast}$. Let
$L^+$ ($R^+$) denote the subgroup of
$\underline{CF}^+(Y,\mathfrak{s};\Lambda_{\omega})$ generated by
the points which lie to the left (right) of $a_1^{\ast}$.

As proved in \cite[Lemma 5.5 and Lemma 5.6]{OSzAnn2}, $L^+$ is a
subcomplex of
$\underline{CF}^+(Y,\mathfrak{s};\Lambda_{\omega})$, and
$R^+$ is a quotient complex. We have a short exact sequence
\[
0 \longrightarrow L^+ \longrightarrow
\underline{CF}^+(Y,\mathfrak{s};\Lambda_{\omega}) \longrightarrow
R^+ \longrightarrow 0,
\]
which gives rise to a long exact sequence:
\[
\cdots \to  H_{d+i}(L^+) \to
\underline{HF}^+_{d+i}(Y,\mathfrak{s};\Lambda_{\omega})
\longrightarrow  H_{d+i}(R^+) \stackrel{\delta}{\longrightarrow}
 H_{d+i-1}(L^+) \to \cdots.
\]
By Corollary~\ref{cor:high grading}, for sufficiently large $i$,
$\underline{HF}^+_{d+i}(Y,\mathfrak{s};\Lambda_{\omega})=0$.
It follows that for all sufficiently large $N$,
\begin{equation}\label{Eq:ChiHF+}
\chi(\underline{HF}^+_{\leq
d+N}(Y,\mathfrak{s};\Lambda_{\omega}))=\chi( H_{\leq
d+N}(L^+))+ \chi( H_{\leq d+N+1}(R^+)).
\end{equation}
 On the other hand,
define $f_1\co R^+ \to L^+$ to be
\[
f_1([x_i^+,i])=[x_i^-,i-n_z(\phi)]t^{\int_{\phi}\eta}.
\]
Here $\phi$ is the disk connecting $x_i^+$ to $x_i^-$ which is
supported in the winding region corresponding to $\alpha_1$, and
$\eta$ is the cochain (representing $\omega$) used to define
$CF^{\circ}(Y;\Lambda_{\omega})$. Then we have another short exact
sequence:
\[
0 \longrightarrow \mathrm{ker}f_1 \longrightarrow R^+
\stackrel{f_1}{\longrightarrow} L^+ \longrightarrow 0,
\]
which induces a long exact sequence:
\[
\cdots \to  H_{d+i}(\mathrm{ker}f_1) \longrightarrow
 H_{d+i}(R^+) \stackrel{f_1}{\longrightarrow}
 H_{d+i-1}(L^+) \longrightarrow  H_{d+i-1}(\mathrm{ker}f_1)
\to \cdots.
\]
It is proved in \cite[Section~5.2]{OSzAnn2} that $\mathrm{ker}f_1$
is a finite dimensional graded vector space over $\Lambda$,
and has Euler characteristic $\chi(\mathrm{ker}f_1)=\pm
T(Y,\mathfrak{s})$. So for all sufficiently large $N$,
\begin{equation}\label{Eq:ChiKer}
\chi(\mathrm{ker} f_1)=\chi( H_{\leq d+N}(L^+))+ \chi( H_{\leq
d+N+1}(R^+)).
\end{equation}
Combining equations (\ref{Eq:ChiHF+}), (\ref{Eq:ChiKer}) and
Corollary~\ref{cor:high grading}, we obtain the desired
result.\end{proof}

Having proved Proposition~\ref{prop:eulerchar}, we can use the
same argument as in \cite[Section 3]{Ni2} to prove the following
homological version of Theorem~\ref{main theorem}.

\begin{prop}
\label{prop:homological main thm} Suppose $Y$ is a closed
$3$--manifold, $F \subset Y$ is an embedded torus. Let $M$ be the
$3$--manifold obtained by cutting $Y$ open along $F$. The two
boundary components of $M$ are denoted by $F_{-}$,$F_{+}$. If
there is a cohomology class $\omega \in  H^2(Y;\mathbb{Z})$
 such that $\omega([F]) \neq 0$ and
$\underline{HF}^+(Y;\Lambda_{\omega}) =\Lambda$, then
$M$ is a homology product, namely, the two maps
\[
i_{\pm \ast}\co  H_{\ast}(F_{\pm}) \to  H_{\ast}(M)
\]
are isomorphisms.
\end{prop}

\section{Proof of Theorem~\ref{main theorem}}

In this section, we will prove Theorem~\ref{main theorem}. We will
essentially follow Ghiggini's argument in \cite{Gh}, with little
modifications when necessary.

\begin{rem}\label{rem:smoothness}
Before we get into the proof, we make a remark on the smoothness
of foliations. In \cite{Ga}, the foliations constructed are
smooth, except possibly along torus components of the given taut
surface. In the proof of \cite[Theorem~3.8]{Gh}, one also modifies
a foliation further by replacing a compact leaf $F$ with an
$F\times I$, which is foliated by $F\times t$'s. The new foliation
may not be smooth if $F$ is a torus. However, by
\cite[Proposition~2.9.4]{ET}, the new foliation can be
approximated in $C^0$--topology by smooth weakly semi-fillable
(hence weakly fillable by \cite{El,Et}) contact structures. Hence
one can run the now standard argument as in \cite[Section~41]{KM}
and \cite{OSzGenus} to get the nontriviality of the corresponding
Heegaard Floer homology.
\end{rem}

\begin{lem}\label{lem:annulus}
Conditions are as in Theorem~\ref{main theorem}. Cut $Y$ open
along $F$, the resulting manifold $M$ has two boundary components
$F_+,F_-$. Suppose $c\subset F_+$ is an essential simple closed
curve, then there exists an embedded annulus $A\subset M$, such
that $c$ is one component of $\partial A$, and the other component
of $\partial A$ lies on $F_-$.
\end{lem}
\begin{proof}
First notice that $M$ is a homological product by
Proposition~\ref{prop:homological main thm}. We can glue the two
boundary components of $M$ by a homeomorphism $\psi$ to get a new
manifold $Y_{\psi}$ with $b_1(Y_{\psi})=1$. Let $\gamma \subset
Y-K$ be a closed curve which is Poincar\'e dual to $\omega$. Then
$\gamma$ also lies in $Y_{\psi}$, we still denote its Poincar\'e
dual in $Y_{\psi}$ by $\omega$, and we have $\omega([F]) \neq 0$
in the new manifold $Y_{\psi}$. By Corollary~\ref{cor:reglue},
\[
\underline{HF}^+(Y_{\psi};\Lambda_{\omega}) \cong
\underline{HF}^+(Y;\Lambda_{\omega}) \cong
\Lambda.
\]
$Y_{\psi}$ satisfies all the conditions in Theorem~\ref{main
theorem}, so we can work with $Y_{\psi}$ instead of $Y$.

From now on we assume $b_1(Y)=1$. We also assume that
$\omega([F])>0$, otherwise we can change the orientation of $F$.

If the conclusion of the lemma does not hold, suppose $c=c_+
\subset F_+$ is an essential simple closed curve such that there
does not exist an annulus $A$ as in the statement of the lemma.
Since $M$ is a homology product, we can find a simple closed curve
$c_- \subset F_-$ homologous to $c_+$ in $M$. We fix an arc
$\delta \subset M$ connecting $F_-$ to $F_+$. Let
$\mathcal{S}_m(+c)$ be the set of properly embedded surfaces $S
\subset M$ such that $\partial S=(-c_-) \cup c_+$ and the
algebraic intersection number of $S$ with $\delta$ is $m$.
$\mathcal{S}_m(+c)\ne\emptyset$ since $M$ is a homology product.

For any surface $S\in\mathcal S_m$, its norm $x(S) >0$. Otherwise
one component of $S$ must be an annulus $A$ connecting $c_-$ to
$c_+$, which contradicts our assumption.

By \cite[Lemma~6.4]{Ni1}, when $m$ is sufficiently large, there is
a connected surface $S_1\in\mathcal S_m(+c)$ which gives a taut
decomposition of $M$. If we reverse the orientation of $c$, when
$n$ is sufficiently large, as before there is $S_2\in\mathcal
S_n(-c)$ which gives a taut decomposition of $M$. As in \cite{Ga},
using these two decompositions, one can then construct two taut
foliations $\mathscr G_1,\mathscr G_2$ of $M$, such that $F_-,F_+$
are leaves of them. These two foliations are glued to get two taut
foliations $\mathscr F_1,\mathscr F_2$ of $Y$, such that $F$ is a
leaf of them. Suppose $S$ is a surface in $\mathcal S_0(+c)$, then
$-S\in\mathcal S_0(-c)$. We have
\begin{eqnarray*}
e(\mathscr F_1,S)&=&e(\mathscr F_1,S_1-mF)=\chi(S_1)-m\cdot0<0,\\
e(\mathscr F_2,-S)&=&e(\mathscr F_2,S_2-nF)=\chi(S_2)-n\cdot0<0,
\end{eqnarray*}
where $e(\mathscr F,S)$ is defined in \cite[Definition~3.7]{Gh}.
Thus we conclude that
\begin{equation}\label{Eq:DiffEuler}e(\mathscr F_1,S)\ne
e(\mathscr F_2,S).
\end{equation}

Choose a diffeomorphism  $\phi\co F_+ \to F_-$ such that
$\phi(c_+)=c_-$. Let $Y_{\phi}$ be the $3$--manifold obtained from
$M$ by  gluing $F_+$ to $F_-$ by $\phi$. Decompose $\phi$ as a
product of positive Dehn twists along non-separating curves
$\{c_1,\dots,c_k \}$ on $F$, then $Y_{\phi}$ is obtained from $Y$
by doing ($-1$)--surgeries along these curves. Let $-W$ be the
cobordism obtained by adding $2$--handles to $Y \times I$ along
these curves with $-1$ framing. As in the beginning of this proof,
$\omega$ also denotes an element in $H^2(Y_{\phi};\mathbb Z)$.

 Since
$\omega([F]) \neq 0$, by Corollary~\ref{cor:reglue}, the map
\[
\underline{F}^+_{-W;\mathrm{PD}(\gamma \times I)}\co
\underline{HF}^+(Y;\Lambda_{\omega}) \to
\underline{HF}^+(Y_{\phi};\Lambda_{\omega})
\]
induced by the cobordism $W$ is an isomorphism.

By Remark~\ref{rem:smoothness}, one can approximate the foliations
$\mathscr{F}_1,\mathscr{F}_2$ on $-Y$ by smooth weakly fillable
contact structures $\xi_1,\xi_2$ on $-Y$. We can realize the above
curves $\{c_1,\dots,c_k \}$ to be Legendrian knots in both $\xi_1$
and $\xi_2$. Let $\xi_i'$ be the contact structure on $-Y_{\phi}$
obtained from $(-Y,\xi_i)$ by doing $(+1)$--contact surgeries
along these Legendrian knots. By Proposition~\ref{prop:contact
surgery}
\[
\underline{F}^+_{-W;\mathrm{PD}(\gamma \times
I)}(c^+(-Y,\xi_i;\Lambda_{\omega}))=c^+(-Y_{\phi},\xi_i';\Lambda_{\omega})
\]
for $i=1,2$. The hypothesis $b_1(Y)=1$, $\omega([F])> 0$, the fact
that $\xi_i$ is weakly fillable and
Corollary~\ref{cor:nonvanishing special} force
$c^+(-Y,\xi_i;\Lambda_{\omega})\ne 0$ for $i=1,2$. Since the map
$\underline{F}^+_{-W;\mathrm{PD}(\gamma \times I)}$ is an
isomorphism as stated above,
\[
c^+(-Y_{\phi},\xi_i';\Lambda_{\omega})=\underline{F}^+_{-W;\mathrm{PD}(\gamma
\times I)}(c^+(-Y,\xi_i;\Lambda_{\omega})) \neq 0
\]
for $i=1,2$.

Let $\overline S$ be the closed surface in $Y_{\phi}$ obtained by
gluing the two boundary components of $S$ together. As in
\cite[Lemma 3.10]{Gh}, (\ref{Eq:DiffEuler}) implies that
\[
\langle c_1(\xi_1'),[\overline{S}] \rangle \neq \langle
c_1(\xi_2'),[\overline{S}] \rangle,
\]
so the Spin$^c$ structures $\mathfrak{s}_{\xi_1'}$ and
$\mathfrak{s}_{\xi_2'}$ are different. Therefore the two elements
$$c^+(-Y_{\phi},\xi_1';\Lambda_{\omega}) \in
\underline{HF}^+(-Y_{\phi},\mathfrak{s}_{\xi_1'};\Lambda_{\omega})$$
and
$$c^+(-Y_{\phi},\xi_2';\Lambda_{\omega})  \in
\underline{HF}^+(-Y_{\phi},\mathfrak{s}_{\xi_2'};\Lambda_{\omega})$$
are linearly independent. Hence $c^+(-Y,\xi_1;\Lambda_{\omega})$
and $c^+(-Y,\xi_2;\Lambda_{\omega})$ are also linearly
independent. We hence get a contradiction to the assumption that
$\underline{HF}^+(Y;\Lambda_{\omega})=\Lambda$.
\end{proof}

\begin{proof}[Proof of Theorem~\ref{main theorem}]
Cut $Y$ open along $F$, we get a compact manifold $M$. By
Lemma~\ref{lem:annulus}, we can find two embedded annuli
$A_1,A_2\subset M$, each has one boundary component on $F_-$ and
the other boundary component on $F_+$, and $\partial
A_1\cap\partial A_2$ consists of two points lying in $F_-$ and
$F_+$, respectively. A further isotopy will ensure that $A_1\cap
A_2$ consists of exactly one arc, so a regular neighborhood of
$A_1\cup A_2$ in $M$ is homeomorphic to $(T^2-D^2)\times I$.  Now
the irreducibility of $Y$ implies that $M=T^2 \times I$. This
finishes the proof.
\end{proof}

\section{Proof of Theorem \ref{fiber knot}}

We turn our attention to the proof of Theorem \ref{fiber knot}. We
will use the following lemmas, which are very similar in spirit.

\begin{lem}
\label{tensor laurant serie} Let $A,B$ be $\mathbb{Z}$--graded
abelian groups such that $A_{(i)}$ and $B_{(i)}$ are finitely
generated for each $i \in \mathbb{Z}$. Let $h,v\co A \to B$ be
homogeneous homomorphisms of the same degree $d$. Suppose $h$ has
a right inverse $\iota$, i.e. $h \circ \iota=\mathrm{Id}_B$.
(Hence the degree of $\iota$ is $-d$.) Then the map
$$h+tv\co A \otimes \Lambda \to B \otimes \Lambda$$
is  surjective and $\mathrm{ker}(h+tv) \cong \mathrm{ker}(h)
\otimes \Lambda$ .
\end{lem}
\begin{proof}
Define a map
 $$P=\sum_{j=0}^{\infty} \iota \circ (v\iota)^{j}(-t)^j.$$
It is well-defined by the fact that the composition $v\iota$ is of
degree 0 and the assumption that
 $B_{(i)}$ is finitely generated for each $i \in \mathbb{Z}$. Clearly
 $$(h+tv) \circ P=\mathrm{Id}_{B}.$$
So $h+tv$ is surjective. Define a map $F\co \mathrm{ker}(h)
\otimes \Lambda \to \mathrm{ker}(h+tv)$ by
$$F(a)=\sum_{i=0}^{\infty}(-t \iota \circ v)^i(a).$$
It has a two sided inverse $G\co \mathrm{ker}(h+tv) \to
\mathrm{ker}(h) \otimes \Lambda$ defined by
$$G(b)=(1+t \iota \circ v)b.$$
$F$ and $G$ are also well-defined. They define an isomorphism
between $\mathrm{ker}(h+tv)$  and $\mathrm{ker}(h) \otimes
\Lambda$.
\end{proof}

\begin{lem}
\label{neglect lower terms} Let $A,B$ be $\mathbb{Z}$--graded
abelian groups as in the above lemma such that  their gradings are
bounded from below. Suppose  $f_1\co A \to B$ and $f_2\co A \to B$
are homomorphisms such that $f_1$ is homogeneous, and for every
homogeneous element $a \in A$ each term in $f_2(a)$ has grading strictly lower
than $f_1(a)$. If $f_1$ is surjective, then $f_1+f_2$ is also
surjective. Furthermore, $\mathrm{ker}(f_1+f_2) \cong
\mathrm{ker}(f_1)$.
\end{lem}
\begin{proof}
See \cite[Proposition~5.8]{OSzAnn2}.
\end{proof}

\begin{proof}[Proof of Theorem \ref{fiber knot}]
If the zero surgery $Y_0(K)$ is a torus bundle, take
$\omega=\mathrm{PD}(\mu)  \in H^2(Y_0(K);\mathbb{Z})$ be the
Poincar\'e dual of the meridian of $K$. By Theorem~\ref{thm:hf for
torus bundle},
$$\underline{HF}^+(Y_0(K);\Lambda_{\omega}) \cong \Lambda,$$
and is
supported in a single torsion Spin$^c$ structure $\mathfrak{s}_0$.
For every $m >0$, $\mathfrak{s}_0$ induces unique Spin$^c$
structures on $Y$ and $Y_m(K)$, respectively. We
 denote these Spin$^c$ structures by $\mathfrak{s},\mathfrak s_m$.
According to Theorem \ref{thm:integral surgery} we have the
following exact sequence:
\begin{equation}\label{Seq:0m}
\begin{xymatrix}{
\underline{HF}^+(Y_0(K),\mathfrak{s}_0;\Lambda_{\omega})
\hspace{10pt}\ar[r]& &\hspace{-60pt}HF^+(Y_m(K),\mathfrak{s}_m) \otimes \Lambda\ar[dl]_{F^+}\\
&HF^+(Y,\mathfrak{s}) \otimes \Lambda\:,\ar[ul]&}
\end{xymatrix}
\end{equation}
where the map $F^+$ is induced by the cobordism $W'_m\co Y_m(K)
\to Y$,
\[
F^+=\sum_{\{\mathfrak{t} \in \mathrm{Spin}^c(W'_m),
\mathfrak{t}|_{Y}=\mathfrak{s},\mathfrak{t}|_{Y_m}=\mathfrak{s}_m\}}\underline{F}^+_{W'_m(K),\mathfrak{t};\mathrm{PD}(\mu
\times I)}.
\]
Note $W'_m$ is a cobordism between two rational homology
$3$--spheres, as in \cite[Section~3.1]{OSzGenus}, the  above maps
are related to the untwisted case by the formula
\begin{equation}\label{Eq:F+}
F^+=\pm t^c \cdot \sum_{\{\mathfrak{t} \in \mathrm{Spin}^c(W'_m),
\mathfrak{t}|_{Y}=\mathfrak{s},\mathfrak{t}|_{Y_m}=\mathfrak{s}_m\}}F^+_{W'_m(K),\mathfrak{t}}
\cdot t^{\frac{\langle c_1(\mathfrak{t}) \cup \mathrm{PD}[\mu
\times I],[W'] \rangle}{2}}.
\end{equation}
$F^+$ has two distinguished summands
$\underline{F}^+_{W'_m,\mathfrak{x}_k;\mathrm{PD}(\mu \times I)}$ and
$\underline{F}^+_{W'_m,\mathfrak{y}_k;\mathrm{PD}(\mu \times I)}$. A simple
calculation shows they are homogeneous maps of the same degree,
and their degree is strictly larger than the degree of any other
summand of $F^+$ \cite[Lemma~4.4]{OSzIntSurg}.

Take $m$ to be sufficiently large, by Theorem~\ref{large surgery}, the exact triangle
 (\ref{Seq:0m}) can be identified with
\begin{equation}\label{Seq:F+}
\hspace{-10pt}\begin{xymatrix}{
\hspace{10pt}\underline{HF}^+(Y_0(K),\mathfrak{s}_0;\Lambda_{\omega})
\hspace{5pt}\ar[r]& &\hspace{-65pt}
H_{\ast}(C(\mathfrak{s}){\{\text{max}(i,j) \geq 0\}})
\otimes \Lambda \ar[dl]_{F^+}\\
&\hspace{-10pt} H_{\ast}(C(\mathfrak{s}){\{i \geq 0\}}) \otimes
\Lambda \ar[ul]&}
\end{xymatrix}
\end{equation}
Under this identification, using equation (\ref{Eq:F+}), $F^+$ can
be written as:
$$F^+=h_{\ast} +tv_{\ast}+ \text{lower degree summands}. $$

On the other hand there is a short exact sequence (see
\cite[Corollary 4.5]{OSzKnot}):
$$0 \to C(\mathfrak{s}){\{i \geq 0 \text{ and } j<0\}} \to C(\mathfrak{s}){\{i \geq 0 \text{ or } j \geq 0\}}
\stackrel{p}{\longrightarrow} C(\mathfrak{s}){\{j \geq 0\}} \to
0.$$ It induces a long exact sequence. Since $C(\mathfrak{s}){\{j
\geq 0\}}$ and $C(\mathfrak{s}){\{i \geq 0\}}$ are chain homotopy
equivalent, we get:
\begin{equation}\label{Seq:h*}
\hspace{-5pt}\begin{xymatrix}{ \widehat{HFK}(Y,K,\mathfrak{s},-1)
\ar[rr]& & H_{\ast}(C(\mathfrak{s}){\{\text{max}(i,j) \geq 0\}})
\ar[dl]_{h_{\ast}}\\
&\hspace{-2pt} H_{\ast}(C(\mathfrak{s}){\{i \geq 0\}}).\ar[ul] &}
\end{xymatrix}
\end{equation}
By assumption $Y$ is an $L$--space, so $
H_{\ast}(C(\mathfrak{s}){\{i \geq 0\}}) \cong
\mathbb{Z}[U,U^{-1}]/{U\mathbb{Z}[U]}$. It follows that $h_*$ is
surjective since it is $U$--equivariant and is an isomorphism at
sufficiently large gradings. Moreover, $h_*$ has a right inverse
since its image is a free abelian group.

Using Lemma \ref{tensor laurant serie} and \ref{neglect lower
terms}, $h_*+tv_*$ and $F^+$ are surjective, and
$$\mathrm{ker}(F^+) \cong\mathrm{ker}(h_*+tv_*)\cong \mathrm{ker}(h_{\ast}) \otimes \Lambda.$$
From exact sequences (\ref{Seq:F+}),(\ref{Seq:h*})
$$\underline{HF}^+(Y_0(K),\mathfrak{s};\Lambda_{\omega})\cong\mathrm{ker}(F^+),\quad
\widehat{HFK}(Y,K,\mathfrak{s},-1)\cong\mathrm{ker}(h_*).$$ Hence
$\widehat{HFK}(Y,K,\mathfrak{s},-1)$ has rank $1$. The same
argument shows that for any other Spin$^c$ structure
$\mathfrak{s}'$ on $Y$, $\widehat{HFK}(Y,K,\mathfrak{s}',-1) \cong
0$. So
$$\widehat{HFK}(Y,K,[F],-1)=\oplus_{\mathfrak{r} \in \mathrm{Spin}^c(Y)}\widehat{HFK}(Y,K,\mathfrak{r},-1)$$
is of rank one, hence  $K$ is a fibered knot by \cite{Ni1}.
\end{proof}

\end{document}